\documentclass [reqno,10pt,oneside]{amsart}

\usepackage{amsthm}
\usepackage{amscd}
\usepackage{amsmath,amssymb}
\usepackage{algorithm}
\usepackage[noend]{algorithmic}

\theoremstyle {plain}
\newtheorem {thm}{Theorem}[section]

\newtheorem {lem}[thm]{Lemma}
\newtheorem {cor}[thm]{Corollary}
\theoremstyle {definition}
\newtheorem {defn}[thm]{Definition}

\theoremstyle {remark}
\newtheorem {rem}[thm]{Remark}
\newtheorem {nota}[thm]{Notation}

\newtheorem {exmp}[thm]{Example}

\DeclareMathOperator{\minAss}{minAss}
\DeclareMathOperator{\height}{ht}

\DeclareMathOperator{\LT}{LT}
\DeclareMathOperator{\LC}{LC}
\DeclareMathOperator{\lcm}{lcm}

\newcommand{\F}{{\mathbb F}}
\newcommand{\N}{{\mathbb N}}
\newcommand{\Q}{{\mathbb Q}}

\newcommand{\Z}{{\mathbb Z}}

\newcommand{\gen}[1]{\left\langle #1 \right\rangle}

\newcommand{\singular}{{\sc Singular }}

\begin{document}

\bibliographystyle{alpha}

\title{An Algorithm for Primary Decomposition in Polynomial Rings over the Integers}

\author{Gerhard Pfister}
\address{Gerhard Pfister\\ Department of Mathematics\\ University of Kaiserslautern\\
Erwin-Schr\"odinger-Str.\\ 67663 Kaiserslautern\\ Germany}
\email{pfister@mathematik.uni-kl.de}
\urladdr{http://www.mathematik.uni-kl.de/$\sim$pfister} 

\author{Afshan Sadiq}
\address{Afshan Sadiq\\ Abdus Salam School of Mathematical Sciences\\ GC University\\ 
Lahore\\ 68-B\\ New Muslim Town\\ Lahore 54600\\ Pakistan}
\email{afshanatiq@gmail.com}

\author{Stefan Steidel}
\address{Stefan Steidel\\ Department of Mathematics\\ University of Kaiserslautern\\ 
Erwin-Schr\"odinger-Str.\\ 67663 Kaiserslautern\\ Germany}
\email{steidel@mathematik.uni-kl.de}
\urladdr{http://www.mathematik.uni-kl.de/$\sim$steidel} 


\thanks{Part of the work was done at ASSMS, GCU Lahore -- Pakistan.}

\date{\today}

\maketitle

\begin{abstract}
We present an algorithm to compute a primary decomposition of an ideal in 
a polynomial ring over the integers. For this purpose we use
algorithms for primary decomposition in polynomial rings over the rationals
resp. over finite fields, and the idea of Shimoyama--Yokoyama resp.
Eisenbud--Hunecke--Vasconcelos to extract primary ideals from pseudo--primary
ideals. A parallelized version of the algorithm is implemented in {\sc Singular}. 
Examples and timings are given at the end of the article.
\end{abstract}

\section{Introduction} \label{secIntro}

Algorithms for primary decomposition in $\Z[x_1,\ldots,x_n]$ have been developed by
Seidenberg (cf. \cite{Se}), Gianni, Trager, Zacharias (cf. \cite{GTZ}) and Ayoub (cf. \cite{A}). 
Within this article we present a slightly different approach which mainly uses primary 
decomposition in polynomial rings over a field and therefore seems to be much more efficient. 
In particular, it uses primary decomposition in $\Q[x_1,\ldots,x_n]$ resp. $\F_p[x_1,\ldots,x_n]$ 
as well as the computation of the minimal associated primes of an ideal in 
$\F_p[x_1,\ldots,x_n]$\footnote{One can choose one of the modern algorithms, cf. \cite{DGP}, 
\cite{EHV}, \cite{GTZ}, \cite{SY}.}, pseudo--primary decomposition\footnote{An ideal is called 
\emph{pseudo--primary} if its radical is prime, cf. \cite{EHV}, \cite{SY}.}, and the extraction of 
the primary components. The essential difference compared to the corresponding algorithm
proposed in \cite{GTZ} is as follows: the primary decomposition of an ideal $I$ in $\Z[x_1,\ldots,x_n]$
with $I \cap \Z = \gen q$ such that $q \neq 0$ is obtained by computing the minimal associated
prime ideals of $I\F_p[x_1,\ldots,x_n]$ for all primes $p$ dividing $q$ and extracting subsequently
the primary ideals.

Let $x = \{x_1,\ldots,x_n\}$ always denote a set of indeterminates and let $I\subseteq \Z[x]$ 
be an ideal. We use the following known facts from commutative algebra for our algorithm:
\begin{enumerate}
\item[(1)] If $I\cap \Z=\langle 0\rangle$, then there exists an $h\in \Z$ such
           that $I:h=I\Q[x]\cap \Z[x]$ and $I=(I:h)\cap \langle I, h\rangle$
           (cf. \cite{Se}, Theorem 2).
\item[(2)] If $I\cap \Z=\langle 0\rangle$ and $I\Q[x] = \overline Q_1 \cap \ldots
           \cap \overline Q_s$ is an irredundant primary decomposition with $\overline P_i
           = \sqrt{\overline Q_i}$, then $I\Q[x] \cap \Z[x] = (\overline Q_1 \cap \Z[x])
           \cap \ldots \cap (\overline Q_s \cap \Z[x])$ is an irredundant primary 
           decomposition and $\overline P_i \cap \Z[x] = \sqrt{\overline Q_i \cap \Z[x]}$
           (cf. \cite{Se}, Theorem 3).
\item[(3)] If $I\cap \Z=\langle q\rangle$ such that $q\neq 0$ and $q=p_1^{\nu_1}\cdots
           p_r^{\nu_r}$ with $p_1, \ldots, p_r$ pairwise different primes, then
           $I=\bigcap_{i=1}^r \langle I, p_i^{\nu_i}\rangle$.
\item[(4)] If $I\cap \Z=\langle p^\nu\rangle$ for some prime $p$ and $\overline{P}_1,
           \ldots, \overline{P}_s$ are the minimal associated primes of $I\F_p[x]$, then
           the canonical liftings\footnote{Choose generators in $\F_p[x]$ and lift the coefficients 
           to non--negative integers smaller than $p$.} $P_1, \ldots, P_s$ to $\Z[x]$ are the 
           minimal associated primes of $I$. 
           
           If $\nu = 1$ let $I\F_p[x] = \overline Q_1 \cap \ldots \cap \overline Q_s$ be an irredundant
           primary decomposition with associated primes $\overline P_1,\ldots, \overline P_s$ and 
           $Q_1,\ldots,Q_s,P_1,\ldots,P_s$ be the canonical liftings to $\Z[x]$. Then $I=Q_1 \cap \ldots
           \cap Q_s$ is an irredundant primary decomposition with associated primes $P_1,\ldots, P_s$. 
\end{enumerate}

The following result can easily be adapted to $\Z[x]$.

\begin{enumerate}
\item[(5)] If $P$ is a minimal associated prime of $I$, then $I+P^m$ is a
           pseudo--primary component of $I$ for a suitable $m \in \N$, i.e. the equidimensional 
           part of $I+P^m$ is the primary component of $I$ associated to $P$. For any $m$ let 
           $Q_m$ be the equidimensional part of $I+P^m$. $Q_m$ is a primary component of $I$ 
           with associated prime $P$ if $Q_m = I\Z[x]_P \cap \Z[x]$ (cf. \cite{EHV}).

           Alternatively we can compute a separator\footnote{We call $s$ a separator 
           of $I$ w.r.t. $P$ if $s\notin P$ and $s$ is contained in all other minimal associated 
           primes of $I$.} $s$ of $I$ w.r.t. $P$ and obtain by $I:s^\infty$ a pseudo--primary 
           component of $I$ (cf. \cite{SY}).
\item[(6)] If $Q_1, \ldots, Q_s$ are the primary components of $I$
           associated to the minimal associated prime ideals and $J=Q_1\cap\ldots \cap
           Q_s$, then there exists a natural number $m$ such that $I=J\cap (I+(I:J)^m)$.
\end{enumerate}

Consequently, by applying (1)--(6), we can reduce the computation of the primary 
decomposition in $\Z[x]$ to the computation of the primary decomposition in $\Q[x]$, the
computation of the minimal associated primes in $\F_p[x]$, and the extraction
of the primary components in $\Z[x]$. In this connection, the extraction has to be generalized 
to polynomial rings over principal ideal domains (cf. Lemma \ref{lemExtraction}). In section 
\ref{secBasDefRes} we state the results used in the algorithm, whereupon in section \ref{secAlg} 
we explain our algorithm which has been implemented in {\sc Singular} in a parallel version. 
Finally we give some examples and the corresponding timings in section \ref{secExTime}.

\section{Basic definitions and results} \label{secBasDefRes}

\begin{defn}
Let $I\subseteq \Z[x]$ be an ideal and $>$ be a monomial ordering on $\Z[x]$. A subset 
$G \subseteq I$ is called a \emph{Gr\"obner basis} of $I$ w.r.t. $>$ if the leading ideal of
$G$ equals the leading ideal of $I$. $G$ is called a \emph{strong Gr\"obner basis} if for 
all $f\in I$ there exists a $g\in G$ such that $\LT(g)|\LT(f)$.\footnote{We use the notations 
of \cite{GP} for the basics of Gr\"obner bases. Especially $\LT(f)$ denotes the leading term 
(leading monomial with leading coefficient) of $f$ w.r.t. the ordering $>$. The theory of (strong)
Gr\"obner bases over principal ideal domains can be found in \cite{AL}, section 4.5.}
\end{defn}

\begin{lem} \label{lemGBQtoZ}
Let $G=\{g_1, \ldots, g_k\} \subseteq \Z[x]$ and $I=\langle G \rangle\Z[x]$. 
Assume that $I \cap \Z=\gen 0$ and $G$ is a Gr\"obner basis of $I\Q[x]$ 
w.r.t. some ordering.
Let $h = \lcm(\LC(g_1),\ldots, \LC(g_k))$ be the least common multiple of the 
leading coefficients of $g_1,\ldots,g_k$. Then $I\Q[x] \cap \Z[x]=I:h^\infty$.
Moreover, if $I:h^\infty = I:h^m$ for some natural number $m$, then $I = (I:h^m) \cap \gen{I,h^m}$.
\end{lem}

The proof of Lemma \ref{lemGBQtoZ} is similar to the corresponding proof for
polynomial rings over a field (cf. \cite{GP}, Proposition 4.3.1).

\begin{rem}
The saturation $I:h^\infty$ can be computed in $\Z[x]$ similarly to the case of a polynomial
ring over a field by computing a Gr\"obner basis of $\gen{I,Th-1}\Z[x,T]$ w.r.t. an 
elimination ordering for $T$: $$I:h^\infty = \gen{I,Th-1}\Z[x,T] \cap \Z[x].$$ 
A natural number $m$ satisfying $I:h^\infty = I:h^m$ can be found by computing the 
normal form of $h^lg$ w.r.t. $I$ for each generator $g$ of $I:h^\infty$ and increasing 
$l \in \N$. More precisely, if the normal form of $h^lg$ w.r.t. $I$ is zero for each generator
$g$ of $I:h^\infty$ then $h^l \cdot (I:h^\infty) \subseteq I$, i.e. $I:h^\infty = I:h^l$.
\end{rem}

\begin{lem}[cf. \cite{SY}] \label{lemSeparator}
Let $I\subseteq \Z[x]$ be an ideal with more than one minimal associated prime, $P$ 
a minimal associated prime and $s \notin P$ a separator, i.e. $s$ is contained in all 
other minimal associated primes of $I$. 
Then $I:s^\infty$ is a pseudo--primary component of $I$, and
$s$ can be chosen as $$\prod_{\substack{Q \neq P \\ Q \in \minAss(I)}} s_Q$$ 
where $s_Q$ is an element of a Gr\"obner basis of $Q$ which is not in $P$.
\end{lem}

\begin{lem}[Extraction Lemma, cf. \cite{GTZ}] \label{lemExtraction}
Let $I=Q\cap J$ be pseudo--primary with $\sqrt{I}=P$ and $Q$ be $P$--primary with 
$\height(Q)<\height(J)$. Let $P\cap \Z=\langle p \rangle$ for some prime $p$ and $u\subset x$ 
be a maximal independent set of variables for $\overline P=P\F_p[x]$\footnote{$u \subset x$ is
called a \emph{maximal independent set} of variables for $\overline P \subseteq \F_p[x]$ if 
$\overline P \cap \F_p[u] = \gen 0$ and $\# u = \dim(\F_p[x]/ \overline P)$; cf. \cite{GP}.}. 
Let $R:=\Z[u]_{\langle p\rangle}$, then the following hold:
\begin{enumerate}
\item[(1)] $I R[x \smallsetminus u]\cap\Z[x]=Q$
\item[(2)] Let $G$ be a strong Gr\"obner basis of $I$ w.r.t. a block 
           ordering satisfying $x \smallsetminus u\gg u$. Then $G$ is a strong Gr\"obner
           basis of $IR[x\smallsetminus u]$ w.r.t. the induced ordering 
           for the variables $x\smallsetminus u$. 
\item[(3)] Let $G=\{g_1, \ldots, g_k\}$ be as in (2), $\LT_{R[x\smallsetminus u]}(g_i) = 
           p^{\nu_i}a_i(x\smallsetminus u)^{\beta_i}$ with $a_i\in \Z[u] \smallsetminus 
           \langle p\rangle$ for $i = 1,\ldots,k$, and $h= \lcm(a_1,\ldots,a_k)$. Then 
           $IR[x\smallsetminus u]\cap\Z[x]=I:h^\infty$.
\end{enumerate}
\end{lem}

\begin{proof} \
\begin{enumerate}
\item[(1)] Let $K=\sqrt{J}$ and $\overline{K}=K\F_p[x]$ then $\overline{K}\supsetneq 
           \overline{P}=P\F_p$. This implies that $\overline{K}\cap\F_p[u]\neq \gen 0$ since
           $u \subset x$ is maximally independent for $\overline P$ and therefore $K\cap 
           (\Z[u]\smallsetminus\langle p\rangle)\neq \emptyset$. 
           Thus it holds $JR[x\smallsetminus u]=R[x\smallsetminus u]$. Finally, 
           because $Q$ is primary, we obtain $I R[x\smallsetminus u]\cap\Z[x]=
           Q R[x\smallsetminus u] \cap \Z[x]=Q$.
\item[(2)] Let $f\in I R[x\smallsetminus u]$ and choose $s\in 
           \Z[u]\smallsetminus\langle p\rangle$ such that $sf\in I$. Since $G$ is a 
           strong Gr\"obner basis of  $I$ there exists a $g\in G$ such that 
           $\LT_{\Z[x]}(g)\mid\LT_{\Z[x]}(sf)$. As a polynomial in $x 
           \smallsetminus u$ with coefficients in $R$, the element $sf$ can be written 
           as $sf=p^\nu a(x\smallsetminus u)^\alpha + (\text{terms in $x\smallsetminus u$ 
           of smaller order})$ with $a\in \Z[u]\smallsetminus \langle p\rangle$. If $p^\tau$ is 
           the maximal power of $p$ dividing the leading coefficient $\LC_{\Z[x]}(g)$ of $g$ 
           then $\tau\leq \nu$ since $\LT_{\Z[x]}(sf)=p^\nu\LT_{\Z[x]}(a)(x\smallsetminus 
           u)^\alpha$. Now we can write $g$ as an element of $R[x\smallsetminus u]$ w.r.t.
           the corresponding ordering, i.e. $g=p^\mu b(x\smallsetminus u)^\beta + (\text{terms in 
           $x\smallsetminus u$ of smaller order})$ with $b\in \Z[u]\smallsetminus \langle 
           p\rangle$ and $\mu\leq \tau \leq \nu$. By definition we have $\LT_{R[x\smallsetminus u]}(g)
           =p^\mu b(x\smallsetminus u)^\beta$ resp. $\LT_{R[x\smallsetminus  u]}(f)=p^\nu 
           \frac{a}{s}(x\smallsetminus u)^\alpha$ and on the other hand it holds
           $\LT_{\Z[x]}(g)=p^\mu \LT_{\Z[x]}(b)(x\smallsetminus u)^\beta$ resp. $\LT_{\Z[x]}(sf)
           =p^\nu\LT_{\Z[x]}(a)(x\smallsetminus u)^\alpha$. Thus the assumption 
           $\LT_{\Z[x]}(g) \mid \LT_{\Z[x]}(sf)$ implies $(x\smallsetminus u)^\beta \mid (x 
           \smallsetminus u)^\alpha$ and consequently
           $\LT_{R[x\smallsetminus u]}(g) \mid \LT_{R[x \smallsetminus u]}(f)$. 
           This proves (2).
\item[(3)] Follows from (2) similarly to the proof for fields (cf. \cite{GTZ}, \cite{GP}).
\end{enumerate}
\end{proof}

The following Lemma is a consequence of the Lemma of Artin--Rees (cf. \cite{GP}).

\begin{lem} \label{lemIntPrimaryComp}
Let $I\subseteq\Z[x]$ be an ideal and $J$ the intersection of all primary
components of $I$ associated to the minimal prime ideals of $I$. Then there exists a
natural number $m$ such that $I=J\cap(I+(I:J)^m)$.
\end{lem}

\begin{nota}
Given an ideal $I \subseteq \Z[x]$ we can always choose a finite set of polynomials
$F_I = \{f_1,\ldots, f_k\}$ such that $I = \gen{F_I}$ and we denote $F_I^{(m)} := \{f_1^m,
\ldots, f_k^m\}$ for $m \in \N$.
\end{nota}

\begin{cor}
With the assumptions and notations of Lemma \ref{lemIntPrimaryComp} there exists 
a natural number $m$ such that $I=J\cap(I+\langle F_{I:J}^{(m)} \rangle)$. 
\end{cor}

\begin{proof}
Due to Lemma \ref{lemIntPrimaryComp} there exists an $m$ such that $I=J\cap(I+(I:J)^m)$. 
Now we have $I\subseteq J \cap (I+\langle F_{I:J}^{(m)} \rangle) \subseteq J \cap (I+(I:J)^m)
= I$ and therefore $I=J\cap(I+\langle F_{I:J}^{(m)} \rangle)$.
\end{proof}

\begin{rem}
The corollary is very important from a computational point of view because 
$\langle F_{I:J}^{(m)} \rangle$ has fewer generators than $(I:J)^m$.
\end{rem}

\section{The algorithms} \label{secAlg}

In this section we present the algorithm to compute a primary decomposition 
of an ideal in a polynomial ring over the integers by applying the results of section 
\ref{secBasDefRes} resp. the introduction (section \ref{secIntro}). 

\vspace{0.2cm}

Algorithm \ref{algPrimdecZ} computes the primary decomposition of an ideal in 
$\Z[x]$\footnote{The corresponding procedures are implemented in \singular in the 
library \texttt{primdecint.lib}.} with the aid of algorithms \ref{algSeparatorsZ} and 
\ref{algExtractZ} which we introduce subsequently in detail.

\begin{rem}
Algorithm \ref{algPrimdecZ} can easily be parallelized by computing  - depending on the prime
factorization $q = p_1^{\nu_1} \cdots p_r^{\nu_r}$ of $q$ where $\gen q = I \cap \Z$ - either the 
primary decomposition or the set of minimal associated primes in positive characteristic in parallel.
If $\nu_i = 1$ we have to compute the primary decomposition whereas, if $\nu_i > 1$, we have to
compute the minimal associated primes of $I\F_{p_i}[x]$ in $\F_{p_i}[x]$. These $r$ computations in 
positive characteristic are independent from each other such that they can also run separately in 
parallel on at most $r$ processors if available.
\end{rem}

\pagebreak

\begin{algorithm}
\caption{\textsc{primdecZ}} \label{algPrimdecZ}
\begin{algorithmic}
\REQUIRE $F_I = \{f_1, \ldots, f_k\}$, $I=\langle F_I \rangle\Z[x]$, optional: a 
         test ideal $T$.
\ENSURE  $L := \{(Q_1, P_1), \ldots, (Q_s, P_s)\}$, $I=Q_1\cap \ldots \cap Q_s$
         irredundant primary decomposition with $P_i=\sqrt{Q_i}$.
\vspace{0.1cm}
\IF{$T$ is not given in the input}
\STATE $T:=\langle 1 \rangle$;
\ENDIF
\STATE $G:=$ strong Gr\"obner basis of $I$;
\STATE $q:=$ generator of $I\cap \Z$;\footnotemark
\IF{$q=0$}
\STATE compute $h\in\Z$ such that $I:h=I\Q[x]\cap\Z[x]$;\footnotemark
\STATE compute $\overline{Q}_1, \ldots, \overline{Q}_s$, an irredundant
       primary decomposition of $I\Q[x]$ and $\overline{P}_i=\sqrt{\overline{Q}_i}$ 
       the associated primes;
\STATE compute $Q_i = \overline{Q}_i\cap\Z[x]$, $P_i=\overline{P}_i\cap\Z[x]$;\footnotemark
\STATE $L:=\{(Q_1, P_1), \ldots, (Q_s, P_s)\}$;
\STATE $M := $ \textsc{primdecZ}$(\langle I, h\rangle)$ \& remove redundant primary ideals from $M$;
\RETURN $L\cup M$;
\ELSE 
\STATE compute $q=p_1^{\nu_1}\ldots p_r^{\nu_r}$, the prime factorization of $q$; 
\FOR{$i=1,\ldots,r$}
\IF{$\nu_i=1$}
\STATE compute $\overline{L}_i=\{(\overline{Q}_1^{(i)}, \overline{P}_1^{(i)}), \ldots, 
       (\overline{Q}_{s_i}^{(i)}, \overline{P}_{s_i}^{(i)})\}$, the primary decomposition of 
       $I\F_{p_i}[x]$;
\STATE $L_i := \{(Q_1^{(i)}, P_1^{(i)}), \ldots, (Q_{s_i}^{(i)}, P_{s_i}^{(i)})\}$, the lifting of 
       $\overline{L}_i$ to $\Z[x]$;\footnotemark
\ELSE
\STATE compute $\overline{A}_i=\{\overline{P}_1^{(i)},\ldots, \overline{P}_{s_i}^{(i)}\}$,
       the set of minimal associated primes of $I\F_{p_i}[x]$ and independent sets of 
       variables $\overline{u}_1^{(i)}, \ldots, \overline{u}_{s_i}^{(i)}$ for 
       $\overline{P}_1^{(i)}, \ldots, \overline{P}_{s_i}^{(i)}$;
\STATE $A_i:=\{P_1^{(i)},\ldots, P_{s_i}^{(i)}\}$, the lifting of $\overline{A}_i$ to $\Z[x]$;
\FOR{$j=1,\ldots,s_i$}
\STATE $Q_j^{(i)} := $ \textsc{extractZ}$(I, A_i, P_j^{(i)}, \overline{u}_j^{(i)})$;
\STATE $L_i := \{(Q_1^{(i)}, P_1^{(i)}), \ldots, (Q_{s_i}^{(i)}, P_{s_i}^{(i)})\}$;
\ENDFOR
\ENDIF
\ENDFOR
\STATE $L := L_1 \cup\ldots\cup L_r$;
\STATE compute $J$, the intersection of all primary ideals in $L$ and $T$;
\IF{$J=I$} 
\RETURN $L$;
\ENDIF
\STATE compute $F_{I:J}$ such that $\langle F_{I:J} \rangle = I:J$;
\STATE compute $m$ such that $J \cap (I+\langle F_{I:J}^{(m)} \rangle) = I$;
\STATE $M := $ \textsc{primdecZ}$(I+\langle F_{I:J}^{(m)} \rangle, J)$ \& remove redundant primary 
               ideals from $M$; 
\RETURN $L\cup M$;
\ENDIF
\end{algorithmic}
\end{algorithm}
\footnotetext[8]{$q$ is either $0$ or the unique element in $G$ of degree $0$.}
\footnotetext[9]{$h$ is a suitable power of the least common multiple of all leading 
                 coefficients of elements in $G$, cf. Lemma \ref{lemGBQtoZ}.}
\footnotetext[10]{$Q_i$ resp. $P_i$ are primary resp. prime due to (4) of the introduction (section 
                 \ref{secIntro}).}
\footnotetext[11]{If $I=\langle F_I \rangle \subseteq \F_p[x]$ then its lifting is obtained by 
                 $\langle p, F_I \rangle$ with the canonical lifting of $F_I$.}       

\pagebreak

The algorithm to compute the separators is based on Lemma \ref{lemSeparator}.

\begin{algorithm}
\caption{\textsc{separatorsZ}} \label{algSeparatorsZ}
\begin{algorithmic}
\REQUIRE $B$ a list of prime ideals generated by a Gr\"obner basis w.r.t. some ordering,
         not contained in each other, $P\in B$.
\ENSURE Polynomial $s$ such that $s\not\in P$, $s\in Q$ for all $Q\in B\backslash\{P\}$.
\vspace{0.1cm}
\FOR{$Q \in B \backslash \{P\}$}
\STATE choose $s_Q$ in the Gr\"obner basis of $Q$ such that $s_Q \notin P$;
\ENDFOR
\RETURN $\prod_{Q \in B \backslash \{P\}} s_Q$;
\end{algorithmic}
\end{algorithm}

The algorithm to extract the primary component from the pseudo--primary component
is based on the Extraction Lemma \ref{lemExtraction}.

\begin{algorithm}
\caption{\textsc{extractZ}} \label{algExtractZ}
\begin{algorithmic}
\REQUIRE $I \subseteq \Z[x]$ an ideal, $B$ the list of minimal associated primes of $I$, 
         $P\in B$ with $P\cap \Z = \gen p$ for some prime $p$, $u\subset x$ an 
         independent set of variables for $P\F_p[x]$.
\ENSURE  The primary component $Q$ of $I$ associated to $P$.
\vspace{0.1cm}
\STATE $s :=$ \textsc{separatorsZ}$(P, B)$;
\STATE $I=I:s^\infty$;
\STATE compute $G=\{g_1,\ldots,g_k\}$, a strong Gr\"obner basis of $I$ w.r.t. a block 
       ordering satisfying $x \smallsetminus u \gg u$;
\STATE compute $\{a_1,\ldots,a_k\}$ such that 
       $\LC_{\Z[u]_{\langle p \rangle}[x \smallsetminus u]}(g_i)= p^{\nu_i} \cdot a_i$ with
       $a_i \in \Z[u] \smallsetminus \gen p$;
\STATE compute $h=\lcm(a_1,\ldots,a_k)$, the least common multiple of $a_1,\ldots,a_k$; 
\RETURN $I:h^\infty$;
\end{algorithmic}
\end{algorithm}

\begin{exmp}
Consider $I=\langle 9, 3x, 3y\rangle$, $P=\langle 3\rangle$, $u=\{x,y\}$ and $B=\{P\}$ in
$\Z[x,y]$. Then we obtain $s=1$, $h=xy$ and thus $I:h^\infty=\langle 3 \rangle$.
\end{exmp}

\section{Examples and timings} \label{secExTime}

In this section we provide examples on which we time the algorithm \texttt{primdecZ} (cf. 
section \ref{secAlg}) and its parallelization  implemented in \textsc{Singular}. 
Timings are conducted by using the 32-bit version of \singular{3-1-2} on an AMD Opteron 
6174 with $48$ CPUs, 800 MHz each, 128 GB RAM under the Gentoo Linux operating system. 
All examples are chosen from The SymbolicData Project (cf. \cite{G}). 

\begin{rem}
The parallelization of our algorithm is attained via multiple processes organized by \singular library 
code. Consequently a future aim is to enable parallelization in the kernel via multiple threads.
\end{rem}

\begin{rem}
In \singular one can compute Gr\"obner bases not only over fields but also over the rings $\Z$
and $\Z/m\Z$ (resp. $\Z/2^l\Z$ as a special case of $\Z/m\Z$). For the integers the implementation is 
based on the theory for Gr\"obner bases over integral domains as introduced by Adams and 
Loustaunau (cf. \cite{AL}, chapter 4). For factor rings further theory needed to be developed by
the \textsc{Singular}-Team in Kaiserslautern (cf. \cite{GSW}, \cite{W}). Details about the corresponding
implementation are presented by Wienand (cf. \cite{W}, chapter 3).
\end{rem}

\pagebreak

We choose the following examples:

\begin{exmp} \label{ex1}
Coefficients: \texttt{integer}, ordering: \texttt{dp}\footnote{\emph{Degree reverse lexicographical ordering:} 
Let $x^\alpha, x^\beta$ be two monomials in $x$, i.e. $\alpha, \beta \in \N^n$. $x^\alpha >_{dp} x^\beta \, :
\Longleftrightarrow \, \deg(x^\alpha) > \deg(x^\beta)$ or $(\deg(x^\alpha) = \deg(x^\beta)$ and $\exists \, 1 
\leq i \leq n: \; \alpha_n = \beta_n, \ldots, \alpha_{i-1} = \beta_{i-1}, \alpha_i < \beta_i)$, where $\deg(x^\alpha) 
= \alpha_1 + \ldots + \alpha_n$; cf. \cite{GP}.}, \texttt{Gerdt-93a.xml} (cf. \cite{G}) considered 
with another integer generator $2 \cdot 3 \cdot 5 \cdot 13 \cdot 17 \cdot 181$.
\end{exmp}

\begin{exmp} \label{ex2}
Coefficients: \texttt{integer}, ordering: \texttt{dp}, \texttt{Gerdt-93a.xml} (cf. \cite{G}) considered 
with another integer generator $2 \cdot 3 \cdot 5 \cdot 13 \cdot 17 \cdot 31 \cdot 181$.
\end{exmp}

\begin{exmp} \label{ex3}
Coefficients: \texttt{integer}, ordering: \texttt{dp}, \texttt{Gerdt-93a.xml} (cf. \cite{G}) considered 
with another integer generator $2 \cdot 3 \cdot 5 \cdot 13 \cdot 17 \cdot 31 \cdot 37 \cdot 181$.
\end{exmp}

\begin{exmp} \label{ex4}
Coefficients: \texttt{integer}, ordering: \texttt{dp}, \texttt{Steidel\_6.xml} (cf. \cite{ES}) considered
with another integer generator $2 \cdot 3 \cdot 5 \cdot 7 \cdot 11 \cdot 13 \cdot 17 \cdot 19 \cdot 23$.
\end{exmp}

\begin{exmp} \label{ex5}
Coefficients: \texttt{integer}, ordering: \texttt{dp}, \texttt{Steidel\_6.xml} (cf. \cite{ES}) considered
with another integer generator $2 \cdot 3^2 \cdot 5 \cdot 7^3 \cdot 11 \cdot 13 \cdot 17 \cdot 19 \cdot 23$.
\end{exmp}

\begin{exmp} \label{ex6}
Coefficients: \texttt{integer}, ordering: \texttt{dp}, \texttt{Gonnet-83.xml} (cf. \cite{BGK}) considered
with another integer generator $2 \cdot 3 \cdot 5 \cdot 7 \cdot 11 \cdot 13 \cdot 17 \cdot 19 \cdot 23$.
\end{exmp}

\begin{exmp} \label{ex7}
Coefficients: \texttt{integer}, ordering: \texttt{dp}, \texttt{Gonnet-83.xml} (cf. \cite{BGK}) considered
with another integer generator $2^2 \cdot 3^2 \cdot 5 \cdot 7 \cdot 11 \cdot 13 \cdot 17 \cdot 19 \cdot 23$.
\end{exmp}

Table \ref{tabPrimdecZ} summarizes the results where $\texttt{primdecZ}^*(k)$ denotes 
the parallelized version of the algorithm using $k$ processes. All timings are given in seconds.

\begin{table}[hbt]
\begin{center}
\begin{tabular}{|r|r|r|r|r|}
\hline
Example & \texttt{primdecZ} & $\texttt{primdecZ}^*(2)$ & $\texttt{primdecZ}^*(3)$ & $\texttt{primdecZ}^*(4)$\\
\hline \hline
\ref{ex1} & \hspace{1.5cm} 604 & \hspace{1.5cm} 383 & \hspace{1.5cm} 339 & \hspace{1.5cm} 249 \\ \hline
\ref{ex2} & 757 & 480 & 392 & 350 \\ \hline
\ref{ex3} & 907 & 542 & 396 & 396 \\ \hline
\ref{ex4} & 17 & 9 & 7 & 4 \\ \hline 
\ref{ex5} & 10 & 6 & 5 & 4 \\ \hline 
\ref{ex6} & 21 & 14 & 10 & 8 \\ \hline 
\ref{ex7} & 39 & 35 & 34 & 31 \\ \hline 
\end{tabular}
\end{center}
\hspace{15mm}
\caption{Total running times for computing a primary decomposition of the considered examples via 
\texttt{primdecZ} and its parallelized variant $\texttt{primdecZ}^*(k)$ for $k = 2,3,4$.} 
\label{tabPrimdecZ}
\end{table}

\section{Acknowledgement}

The authors would like to thank the anonymous referees whose comments led to great improvement
of the paper.


\begin{thebibliography}{999999}
\bibitem[{A}]{A} Ayoub, C.W.: The Decomposition Theorem for Ideals in Polynomial Rings over 
  a Domain. Journal of Algebra 76, 99--110 (1982).
\bibitem[{AL}]{AL} Adams, W.W.; Loustaunau, P.: An Introduction to Gr\"obner Bases. Graduate 
  Studies in Mathematics, Volume 3, American Mathematical Society (1994).
\bibitem[{BGK}]{BGK} Boege, W.; Gebauer, R.; Kredel, H.: Some Examples for Solving Systems of 
  Algebraic Equations by Calculating Groebner Bases. Journal of Symbolic Computation 1, 83--98 
  (1986).
\bibitem[{DGP}]{DGP} Decker, W.; Greuel, G.-M.; Pfister, G.: Primary Decomposition: 
  Algorithms and Comparisons. In: Algorithmic Algebra and Number Theory, Springer, 187--220 
  (1998).
\bibitem[{DGPS}]{DGPS} Decker, W.; Greuel, G.-M.; Pfister, G.; Sch{\"o}nemann, H.:
  \newblock {\sc Singular} {3-1-1} --- {A} computer algebra system for polynomial 
  computations. \newblock {http://www.singular.uni-kl.de} (2010).
\bibitem[{EHV}]{EHV} Eisenbud, D.; Huneke, C.; Vasconcelos, W.: Direct Methods for Primary
  Decomposition. Inventiones Mathematicae 110, 207--235 (1992).
\bibitem[{ES}]{ES} Eisenbud, D.; Sturmfels, B.: Binomial ideals. Duke Mathematical Journal 84 
  (No. 1), 1--45 (1996).
\bibitem[{GP}]{GP} Greuel, G.-M.; Pfister, G.: A \textsc{Singular} Introduction to 
  Commutative Algebra. Second edition, Springer (2007).
\bibitem[{GSW}]{GSW} Greuel, G.-M.; Seelisch, F.; Wienand, O.: The Gr\"obner basis of the ideal of 
  vanishing polynomials. Journal of Symbolic Computation, Article in Press, Corrected Proof, 
  doi:10.1016/j.jsc.2010.10.006, 14 pages (2011). 
\bibitem[{GTZ}]{GTZ} Gianni, P.; Trager, B.; Zacharias, G.: Gr\"obner Bases and Primary
  Decomposition of Polynomial Ideals. Journal of Symbolic Computation 6, 149--167 (1988).
\bibitem[{G}]{G} Gr\"abe, H.-G.: The SymbolicData Project --- Tools and Data for Testing 
  Computer Algebra Software. \newblock {http://www.symbolicdata.org} (2010).
\bibitem[{M}]{M} Monico, C.: Computing the Primary Decomposition of
  zero--dimensional Ideals. Journal of Symbolic Computation 34, 451--459 (2002).
\bibitem[{Sa}]{Sa} Sausse, A.: A New Approach to Primary Decomposition. Journal of Symbolic 
  Computation 11, 1--15 (1996).
\bibitem[{Se}]{Se} Seidenberg, A.: Constructions in a Polynomial Ring over the Ring of 
  Integers. American Journal of Mathematics 100 (No. 4), 685--703 (1978).
\bibitem[{SY}]{SY} Shimoyama, T.; Yokoyama, K.: Localization and Primary Decomposition of 
  Polynomial Ideals. Journal of Symbolic Computation 22, 247--277 (1996).
\bibitem[{W}]{W} Wienand, O.: Algorithms for Symbolic Computation and their Applications.
  Ph.D. Thesis, Kaiserslautern (2011).
\end{thebibliography}
\end{document}